\documentclass[11pt]{article}

\usepackage{epsfig}
\usepackage{graphicx}
\usepackage{amsbsy}
\usepackage{amsmath}
\usepackage{amsfonts}
\usepackage{amssymb}
\usepackage{textcomp}
\usepackage{hyperref}
\usepackage{aliascnt}

\newcommand{\mcm}[3]{\newcommand{#1}[#2]{{\ensuremath{#3}}}} 

\mcm{\tuple}{1}{\langle #1 \rangle}
\mcm{\name}{1}{\ulcorner #1 \urcorner}
\mcm{\Nbb}{0}{\mathbb{N}}
\mcm{\Zbb}{0}{\mathbb{Z}}
\mcm{\Rbb}{0}{\mathbb{R}}
\mcm{\Cbb}{0}{\mathbb{C}}
\mcm{\Qbb}{0}{\mathbb{Q}}
\mcm{\Acal}{0}{\cal A}
\mcm{\Bcal}{0}{\cal B}
\mcm{\Ccal}{0}{\cal C}
\mcm{\Dcal}{0}{\cal D}
\mcm{\Ecal}{0}{\cal E}
\mcm{\Fcal}{0}{\cal F}
\mcm{\Gcal}{0}{\cal G}
\mcm{\Hcal}{0}{\cal H}
\mcm{\Ical}{0}{\cal I}
\mcm{\Jcal}{0}{\cal J}
\mcm{\Kcal}{0}{\cal K}
\mcm{\Lcal}{0}{\cal L}
\mcm{\Mcal}{0}{\cal M}
\mcm{\Ncal}{0}{\cal N}
\mcm{\Ocal}{0}{{\cal O}}
\mcm{\Pcal}{0}{{\cal P}}
\mcm{\Qcal}{0}{{\cal Q}}
\mcm{\Rcal}{0}{{\cal R}}
\mcm{\Scal}{0}{{\cal S}}
\mcm{\Tcal}{0}{{\cal T}}
\mcm{\Ucal}{0}{{\cal U}}
\mcm{\Vcal}{0}{{\cal V}}
\mcm{\Wcal}{0}{{\cal W}}
\mcm{\Xcal}{0}{{\cal X}}
\mcm{\Ycal}{0}{{\cal Y}}
\mcm{\Zcal}{0}{{\cal Z}}
\mcm{\Mfrak}{0}{\mathfrak M}

\mcm{\restric}{0}{\upharpoonright}
\mcm{\upset}{0}{\uparrow}
\mcm{\onto}{0}{\twoheadrightarrow}
\mcm{\smallNbb}{0}{{\small \mathbb{N}}}
\DeclareMathOperator{\preop}{op}
\mcm{\op}{0}{^{\preop}}

{\begin{array}{c}
\setlength{\unitlength}{1em}}%
{\end{array}}

\usepackage{amsthm}

\newcommand{\theoremize}[2]{\newaliascnt{#1}{thm} \newtheorem{#1}[#1]{#2} \aliascntresetthe{#1}}

\theoremstyle{plain}
\newtheorem{thm}{Theorem}[section]
\theoremize{lem}{Lemma}
\theoremize{skolem}{Skolem}
\theoremize{fact}{Fact}
\theoremize{sublem}{Sublemma}
\theoremize{claim}{Claim}
\theoremize{obs}{Observation}
\theoremize{prop}{Proposition}
\theoremize{cor}{Corollary}
\theoremize{que}{Question}
\theoremize{oque}{Open Question}
\theoremize{con}{Conjecture}

\theoremstyle{definition}
\theoremize{dfn}{Definition}
\theoremize{rem}{Remark}
\theoremize{eg}{Example}
\theoremize{exercise}{Exercise}
\theoremstyle{plain}

\usepackage{verbatim}
\usepackage{enumerate}
\usepackage[all]{xy}

\usepackage{subfig}

\title{A Whitney type theorem for surfaces:
\newline characterising graphs with locally planar 
embeddings
}

\author{Johannes Carmesin
\medskip 
\\
  {University of Birmingham}
}

\mcm{\Fbb}{0}{\mathbb{F}}

\usepackage{xr}
\begin{document}

\maketitle

\begin{abstract}
 We prove that for any parameter $r$ an $r$-locally 2-connected graph $G$ embeds $r$-locally 
planarly in a surface if and 
only if a 
certain matroid associated to the graph $G$ is co-graphic.  
 This extends Whitney's abstract planar duality theorem from 1932. 
\end{abstract}

\section{Introduction}

A fundamental question in Structural Graph Theory is how to embed graphs in surfaces.
There are two main lines of research.

Firstly, in the specific embedding problem we have a fixed surface and are interested in embedding 
a given graph in that surface. 
Mohar proved the existence of a linear time algorithm for this problem, solving the algorithmic 
aspect of this problem \cite{mohar1999linear}. This proof was later simplified by Kawarabayashi, 
Mohar and Reed \cite{kawarabayashi2008simpler}. 

Secondly, in the general embedding problem, we are interested in finding for a given graph a surface 
so that the graph embeds in that surface in an optimal way. Usually people used minimum genus as the 
optimality criterion \cite{{kawarabayashi2008simpler},{MoharThomassen},{mohar2001face}}. However, 
in 1989 Thomassen 
showed that with this interpretation, the 
problem would be NP-hard \cite{thomassen1989graph}. Instead of minimum genus, here we use local 
planarity as our optimality criterion -- and provide a polynomial algorithm for the general 
embedding problem. 

While this discrepancy in the algorithmic complexity implies that maximally locally planar 
embeddings cannot always be of minimum genus, 
Thomassen showed that they are of minimum genus if all face boundaries are shorter than 
non-contractible cycles of the embedding \cite{thomassen1990embeddings}.

\vspace{.3cm}

In 1932 Whitney initiated\footnote{While I see Whitney's abstract duality theorem as the 
point of birth of matroid theory, Whitney's axioms for matroids were only published in 1935, 
see \cite{whitney_matroids}. Independently, matroids were introduced by 
Nakasawa  in 1935, see \cite{nakasawa2009axiomatics}.} the systematic study of matroids by proving 
that a graph can be 
embedded in the plane if and only if its cycle matroid is co-graphic 
{\cite[Theorem 29]{{Whitney32}}}. 
Here we work with embeddings of graphs in general surfaces that are 
just \emph{planar 
locally}; meaning that, every cycle of bounded length is generated by the face boundaries.
Analogous to the cycle matroid of a graph $G$, we define the \emph{local matroid} of $G$ 
whose circuits are those generated by cycles of $G$ of bounded length. A graph is \emph{$r$-locally 
2-connected} if for every vertex $v$ the punctured ball $B_{r/2}(v)-v$ is connected.
We prove the following extension of Whitney's duality theorem. 

\begin{thm}\label{new_main}
 Given a parameter $r\in \Nbb\cup\{\infty\}$, an $r$-locally 2-connected graph embeds 
$r$-locally planarly in a surface if and 
only if its $r$-local matroid is co-graphic.  
\end{thm}

The assumption of local 2-connectedness is not essential. Indeed, if we have a graph that is not 
locally 2-connected, one could cut locally at those local cutvertices, embed the graph via 
\autoref{new_main}, and then glue the vertices back together and adjust the embedding near these 
gluing vertices. 

We believe that this is formulated most naturally in terms of \emph{pseudo-surfaces}, which are 
obtained from surfaces by identifying finitely many points. 
While every surface is an example of a pseudo-surface, every 
embedding of a graph in a pseudo-surface can be extended to an embedding in a genuine surface.

\begin{cor}\label{mainthm_whitney_intro}
Given a parameter $r\in \Nbb\cup\{\infty\}$, a graph embeds 
$r$-locally planarly in a pseudo-surface if and 
only if its $r$-local matroid is co-graphic.  
\end{cor}

Checking whether a given matroid is graphic can be checked in quadratic time.\footnote{See 
\cite{geelen2013characterizing} for on overview of algorithms and a short proof that this problem 
is polynomial.} 
So a consequence of \autoref{mainthm_whitney_intro} is that there is a polynomial time algorithm 
that checks 
whether a graph has a locally planar embedding.

\vspace{.15cm}

{\bf Locally planar embeddings.}
The idea of using embeddings that locally look like plane embeddings as a tool has been 
introduced\footnote{according to \cite{mohar1996planar}} by Robertson and Seymour as part of their 
Graph Minors project \cite{robertson1988graph}. 
Thomassen proved that LEW\footnote{According to \cite{thomassen1990embeddings}, LEW 
embeddings were introduced by Hutchingson in \cite{hutchinson1984automorphism}.}
embeddibility, a 
particular type of locally planar embeddibility, can be tested in polynomial time, and in fact such 
embeddings have always minimum genus  \cite{thomassen1990embeddings}. 
Since then locally planar embeddings (and their cousins `face width' and `edge width')
have become an essential part of 
Topological Graph Theory \cite{MoharThomassen}. 

For example, they are used 
in the recent project of three-colouring 
triangle-free graphs on surfaces by  Dvo{\v{r}}{\'a}k, Kr{\'a}l and 
Thomas \cite{dvovrak2009coloring}.
Very often proving a result for locally planar embeddings is already considered a major step worth 
publication \cite{{albertson2001extending}, 
{devos2008locally}, {kawarabayashi2004theorem},   
{kawarabayashi2010star},  {kawarabayashi2007chords}}.
For example DeVos and  Mohar generalised a theorem of Thomassen by proving that graphs with large 
edge-width are 5-list-colourable \cite{devos2008locally}.

\vspace{.15cm}

{\bf Duality in surfaces.}
Abrams and Slilaty have a whole series of papers in which they explore the relationship between 
surface embeddings, dual graphs and matroids. 
Based on foundational work of Zaslavsky \cite{zaslavsky1993projective}, in 
\cite{slilaty2002matroid}, Slilaty describes toridal duality and projective duality in terms of 
lift matroids and bias matroids.
In \cite{slilaty2005cographic} and \cite{abrams2003algebraic}, Abrams and Slilaty characterise 
embeddability of a graph $G$ in the projective plane via representability of the cycle matroid of 
$G$ as a bias matroid of a signed graph. 
In \cite{abrams2006algebraic}, they prove a theorem describing duality in pseudo-surfaces. 
This inspired the first half of the proof of \autoref{construct_embedding}. 
Slilaty \cite{slilaty2005cographic} asked the following: 
`Planarity of graphs precisely determines the intersection of the
class of cographic matroids with the class of graphic matroids. But
what happens when [the graph] $G$ is nonplanar?'

We derive \autoref{mainthm_whitney_intro} as a corollary of a general duality theorem for surfaces, 
\autoref{abstract_duality} below, which provides an answer to Slilaty's question. 
Indeed, given a graph $G$ and a set $S$ of its cycles, an embedding of $G$ in a 
surface is \emph{$S$-facial} if every cycle in $S$ is generated by the face boundaries. 
\begin{eg}
Every embedding is an $S$-facial embedding, where $S$ is the set of cycles generated by the face 
boundaries.  For 
$S$ equal to the cycle space of $G$, $S$-facial embeddings of $G$ are just plane embeddings. 
\end{eg}
In \autoref{abstract_duality} , we show that (under certain necessary conditions) $G$ has an 
$S$-facial embedding if and only if a certain matroid is 
co-graphic. 

\section{Duals of \emph{S}-facial embeddings}

In this section we prove \autoref{duality_in_surfaces} below, which is used in the proof of the 
main result stated in the Introduction. This lemma expresses duality in surfaces in terms of 
matroids.

Fix a field $k$ to be one of the finite fields $\Fbb_2$ or $\Fbb_3$.
\begin{rem}
 To investigate embeddings in general surfaces, take $k=\Fbb_2$. For embeddings in orientable 
surfaces, take $k=\Fbb_3$. 
\end{rem}

Given a graph $G$ embedded in a surface, a cycle of $G$ is \emph{facially generated} if it is a sum 
of face boundaries over the field $k$, where all coefficients of the sum are plus one\footnote{This 
last condition is trivial for $k= \Fbb_2$.}.

\begin{eg}
Cycles that are face boundaries are facially generated.  More generally cycles bounding discs are 
facially generated. The double torus has a cycle that separates 
its two handles; this cycle is facially generated. On the other hand, cycles cutting a handle are 
not facially generated. Contractible cycles are facially generated over $\Fbb_2$. The converse is 
not true as the 
double torus has a cycle that is facially generated (that is homologically trivial over $\Fbb_2$) 
but not 
contractible (that is homotopically trivial). 
\end{eg}

Given a graph $G$ and a subspace $S$ of its cycle space over the field $k$, we say that an 
embedding of the graph $G$ 
is \emph{$S$-facial} if all cycles of the graph $G$ that are in the vector space $S$ are facially 
generated. 

\begin{eg}
 If the vector space $S$ is equal to the cycle space of a graph $G$, then $S$-facial embeddings of 
$G$ are 
simply plane embeddings of $G$. 

Any embedding of a graph $G$ in a surface is $S$-facial, where $S$ is a subspace of the vector 
space generated over $\Fbb_2$ by the face boundaries. If the embedding is in an oriented surface, 
the faces generate the same cycles over the fields $\Fbb_2$ and $\Fbb_3$, and in fact the set $S$ 
generated by the faces is in this case the cycle space of a regular matroid. 
\end{eg}

\begin{rem}
 In an attempt to simplify notation, we will sometimes consider sets of 
$k$-vectors whose coordinates are the edges of a graph $G$, such as the subspace $S$ defined below, 
simply as 
sets of edge sets of $G$. Formally, we identify the set $S$ with the set containing 
the supports of the vectors in $S$. Usually, we indicate this by writing things like `the set 
$S$'.
\end{rem}

Fix a graph $G$, and a subspace $S$ 
of its cycle space, and an $S$-facial embedding $\iota$ of the graph $G$ in a surface $\Sigma$.
A \emph{hole} of the $S$-facial embedding $\iota$ is a face of the embedding whose boundary is 
not in the set $S$. 

\begin{eg}
 $S$-facial embeddings in which all face boundaries are in the set $S$ do not have any holes. 
\end{eg}

An \emph{area} (bounded by $S$) is a sum of faces of the embedding $\iota$ such that the sum of 
these face boundaries is in the vector space $S$ and all coefficients of the sum are plus 
one. 

\begin{eg}
 A cycle is facially generated if and only if it bounds an area. 
\end{eg}

We say that a set of holes is \emph{fenced in} by an area $A$ if it is equal to the 
set 
of all holes contained in the area $A$ (that is more formally, that are contained in the sum for 
the area $A$). 
We say that an area $A$ bounded by $S$ is \emph{minimal} if 
$A$ contains at least one hole and there is 
no other area bounded by $S$ fencing in a proper nonempty subset of the set of holes fenced in 
by $A$.

We say that a set $S$ is \emph{fencing in} for the embedding $\iota$ if the sets of holes 
fenced in by any two minimal areas bounded by $S$ are either identical or disjoint.
A set $S$ that is fencing in for any embedding, we simply call \emph{fencing in}.

\begin{eg}
If the vector space $S$ is equal to the circuit space of the graph $G$, then it is fencing in, as 
there are no holes. 

\autoref{is_fencing_in} below says that for any $r\in \Nbb$ if the vector space $S$ is generated by 
the cycles of $G$ of length at most $r$, 
then it is fencing in.
\end{eg}

\begin{eg}
Here we give an example of a set $S$ that is not fencing in. 
The graph $G$ is the ladder with eight vertices, embedded in the plane as indicated 
in \autoref{fig:ladder}. Let $S$ be the set generated by the two cycles of length six. 
Then all the three cycles of length four of the ladder are holes. And every area contains an even 
number of these $4$-cycles. As for any pair of the three four-cycles, there is an area 
containing them, the set $S$ is not fencing in. 
\end{eg}

   \begin{figure} [htpb]   
\begin{center}
   	  \includegraphics[height=1.5cm]{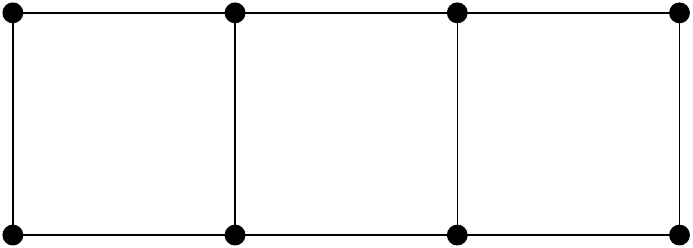}
   	  \caption{An embedding of the ladder with eight vertices in the plane. 
}\label{fig:ladder}
\end{center}
   \end{figure}

We say that a subspace $S$ of the cycle space of $G$ has a \emph{cyclic generating set} if there is 
a set of cycles of $G$ generating $S$.

\begin{eg}
 Trivially, the vector space $S$ generated by the cycles of length at most $r$ has a cyclic 
generating set.
\end{eg}

Given a subspace $S$ of the cycle space of a graph $G$ over the field $k$, the 
\emph{$S$-local} matroid for $G$ is the matroid with ground set 
$E(G)$ 
whose 
circuits are the minimal nonempty supports of elements of $S$. 

\begin{rem}
 By definition, the $S$-local matroid is $k$-represented. 
\end{rem}

\begin{eg}
 In this paper we are mostly interested in $S$-local matroids, where the subspace $S$ is the set of 
circuits generated by circuits of length at most $r$. We refer to this matroid as the 
\emph{$r$-local} matroid of the graph $G$. 
\end{eg}

Given the field $k$, we say that 
an embedding of a graph $G$ is \emph{$k$-admissible} if the sum over all face boundaries of the 
embedding with coefficient plus one is identically zero, when evaluated over the field $k$. 
\begin{eg}
 For $k=\Fbb_2$, every embedding is $k$-admissible. For $k=\Fbb_3$, the $k$ admissible embeddings 
such that all faces are discs are precisely the embeddings in orientable surfaces. 
\end{eg}

\begin{lem}\label{duality_in_surfaces}
 Let $G$ be a graph, $S$ be a nonempty subspace of the cycle space of $G$ over $k$ with a cyclic 
generating set. Let $\iota$ be an $S$-facial $k$-admissible embedding of $G$. 
Assume that $S$ is fencing in for $\iota$. 
Let $H$ be the 
dual graph of the 
embedding $\iota$.

Then there is a quotient of the graph $H$ whose cycle matroid is equal to the dual of the $S$-local 
matroid of $G$. 
\end{lem}

\begin{rem}\label{rem100}
\autoref{duality_in_surfaces} may be regarded as an analogue of the statement that the cycle 
matroids of dual plane graphs are dual to one another.
Indeed, if the embedding $\iota$ does not have any holes, the proof below will show that the cycle 
matroid of the dual graph $H$ is equal to the dual of the $S$-dual matroid for $G$. 
\end{rem}

\begin{proof}[Proof of \autoref{duality_in_surfaces}.]
 We start with the construction of a quotient $Q$ of the dual graph $H$. We say that two holes are 
\emph{related} if they are contained in a common minimal area bounded by $S$. 
\begin{sublem}\label{are_contained}
 Every hole is contained in an area bounded by $S$. 
\end{sublem}

\begin{proof}
By assumption, the set $S$ contains a cycle. As the embedding is $S$-facial, there is an area 
bounded by this cycle of $S$. 
Denote this area by $A$. Let $X$ be the set of all faces of the embedding $\iota$. As the 
embedding $\iota$ is $k$-admissible by assumption, the sum over the face boundaries of $X$ with 
coefficient plus one is 
identically zero over the field $k$. 
As every hole is contained in one of the areas $A$ or $X-A$, every hole is contained in an area.  
\end{proof}

By \autoref{are_contained}, `being related' is a reflexive relation (that is, all holes are related 
to themselves). It is symmetric by definition.  
So as the set $S$ is 
fencing 
in, this relation is transitive, so it defines an equivalence relation on the set of holes.
Each hole is a vertex of the graph $H$. We obtain the quotient $Q$ from the graph $H$ by 
identifying any two vertices of holes that are related. 

It remains to show that the bond matroid of the graph $Q$ is equal to the $S$-local matroid for the 
graph $G$. We denote the bond matroid of the graph $Q$ by $M(Q)$ and the $S$-local matroid by 
$M(S)$. 

\begin{sublem}\label{cycle-lem1}
Every cycle $o$ of $G$ in $S$ is in the circuit space of $M(Q)$.
\end{sublem}

\begin{proof}
As the embedding $\iota$ is $S$-facial, the cycle $o$ of $S$ is facially generated. 
Let $A'$ be an area whose boundary is the cycle $o$. 
Let $A$ be the set of vertices of the graph $H$ contained in the area; meaning that, their faces 
get the coefficient plus one in the area $A'$, and let $B$ be 
the set of other vertices of the graph $H$.
In the graph $H$, the cycle $o$ is the cut consisting of the edges from $A$ to $B$. 
Indeed, even with coefficients taken into account, the cycle $o$ is the sum over the atomic 
cuts of the vertices in $A$. 
Moreover, the 
set of holes fenced in by a minimal area is either completely contained in $A$ or $B$ as the set 
$S$ is fencing in by assumption. 
Hence the cut $o$ of the graph $H$ induces a cut of the quotient $Q$. 
To summarise, the edge set $o$ is a cut of the graph $Q$, and thus in the circuit space of 
the matroid $M(Q)$.
\end{proof}

\begin{sublem}\label{cycle-lem2}
For every atomic cut of $Q$, its characteristic vector is in $S$.
\end{sublem}

\begin{proof}
There are two types of vertices of $Q$, those that are vertices of the graph $H$ and those that 
come from identifying holes from an equivalence class. 
Atomic cuts for vertices of the first type come from faces whose boundary is in the set $S$. Hence 
(characteristic vectors of)
these atomic cuts are in the vector space $S$. 
Now let $q$ be a vertex of the quotient $Q$ coming from identifying the holes fenced in by a 
minimal area bounded by $S$. 
Let $D$ be a minimal area fencing in precisely the holes of $q$. 
Let $o$ be the boundary of the area $D$. 
Let $D=\sum F_i$ be a representation of the area $D$ as a sum of faces $F_i$.
Let $X$ be the sum of the terms $\partial(F_i)$ for faces $F_i$ that are not holes, where 
$\partial(F_i)$ denotes 
the boundary of the face $F_i$. By 
definition of holes, the difference $o-X$ is in the 
vector space $S$. 
The difference $o-X$ is equal to the sum of the face boundaries of holes -- with all coefficients 
being plus one -- in the equivalence class 
for $q$. So the (characteristic vector of the) atomic cut at $q$ is a sum of the 
characteristic vectors of atomic cuts of $H$ for the holes in $q$. So it is a vector in the vector 
space $S$. 
\end{proof}

Now we show that the $k$-represented matroids $M(Q)$ and $M(S)$ are isomorphic. For that we check 
that they 
have the same circuit space. As the set $S$ has a cyclic generating set, the circuit space $S$ of 
the matroid $M(S)$ is generated by a set of cycles of the graph $G$. By \autoref{cycle-lem1}, all 
these cycles are in the circuit space of the matroid 
$M(Q)$. So the circuit space of $M(S)$ is a subspace of the circuit space of the matroid $M(Q)$. 
As the circuit space of the matroid $M(Q)$ is generated by the atomic cuts  of the graph $Q$ (over 
the integers and 
hence also over the field $k$), by 
\autoref{cycle-lem2} the circuit space of $M(Q)$ is a subspace of $S$.
Thus the two matroids $M(Q)$ and $M(S)$ are isomorphic, completing the proof. 

\end{proof}

\section{Constructing embeddings from abstract duals}

In this section, we prove \autoref{construct_embedding} below, which is used in the proof of the 
main result stated in the Introduction. In this lemma we construct an embedding of a graph $G$ in a 
surface using an abstract dual 
graph. 

We say that a subspace $S$ of the cycle space of a graph $G$ is \emph{locally connected} if for 
every vertex $v$ of $G$ there is no vector supported at a non-empty proper subset of the atomic 
cut at $v$ that 
is orthogonal to the vector space $S$.

\begin{eg}
 If $k=\Fbb_2$, $G$ is locally connected if and only if for 
every vertex $v$ of $G$ no non-empty proper subset of the atomic cut at $v$ intersects all sets of 
$S$ evenly.
\end{eg}

\begin{dfn}
 Given a graph $G$ with a vertex $v$ and an integer $s$, the \emph{ball} of radius $s$ around the 
vertex $v$ is the induced subgraph of $G$, whose vertices are those of distance at most $s$ from 
$v$ and without all edges joining two vertices of distance precisely $s$.
Similarly, given a half-integer $s+\frac{1}{2}$, the \emph{ball} of radius $s+\frac{1}{2}$ around 
the 
vertex $v$ is the induced subgraph of $G$, whose vertices are those of distance at most $s$ from 
$v$. We denote the ball of radius $s$ around $v$ by $B_s(v)$. 
Given a parameter $r$, a vertex $v$ is an \emph{$r$-local cutvertex} if the punctured ball 
$B_{r/2}(v)-v$ is disconnected.  
\end{dfn}

\begin{lem}\label{is_loc_con}
Assume $G$ has no $r$-local cutvertex. 
 The subspace of the cycle space of $G$ generated by the cycles of length at most $r$ is locally 
connected. 
\end{lem}

\begin{proof}
Let $S$ denote the subspace of the cycle space of $G$ generated by the cycles of length at most 
$r$. 
 Suppose for a contradiction, there is a vector $\vec{X}$ supported at a non-empty proper subset 
$X$ of an atomic cut of a vertex 
$v$ that is orthogonal to all vectors in the vector space $S$ over the field $k$. 
Let $a$ be a neighbour of the vertex $v$ that is incident with an edge of $X$, and $b$ be a 
neighbour of the vertex $v$ that is incident with an edge incident with $v$ that is not in $X$. 

As the vertex $v$ is no $r$-local cutvertex, the punctured ball $B_{r/2}(v)-v$ contains a path from 
the 
vertex $a$ to the vertex $b$. This path together with the vertex $v$ is a cycle. Denote it by $o$. 
By construction, the cycle $o$ intersects the set $X$ precisely once. (In the case $a=b$, the cycle 
$o$ has 
size two and is chosen to contain a single edge of the set $X$.)
By {\cite[\autoref*{cycle_gen}]{{loc2sepr}}} and 
{\cite[\autoref*{gen}]{{loc2sepr}}}  the 
cycle $o$ is generated by cycles of length at most $r$.
As the cycle $o$ is not orthogonal to the vector $\vec{X}$, one of the 
generating cycles has to be non-orthogonal to the vector $\vec{X}$. As this generating cycle 
is in the set $S$, we get the desired contradiction.  
\end{proof}

\begin{lem}\label{construct_embedding}
  Let $G$ be a $2$-connected graph and $S$ be a subspace of the cycle space of $G$ over $k$.
Assume that $S$ is locally connected. 
Assume that the dual $M^*$ of the $S$-local matroid of $G$ is graphic.

Then the graph $G$ has an $S$-facial $k$-admissible embedding $\iota$ such that the dual graph of 
$\iota$ has a 
quotient whose cycle matroid is $M^*$. 

Moreover, all faces of the embedding $\iota$ are discs. 
\end{lem}

\begin{proof}
 
Let $H$ be a graph representing the dual matroid $M^*$ of the $S$-local matroid of $G$. 
For later reference, we stress that we do pick the graph $H$ such that it does not have a 
cutvertex; this is always possible. We also assume that the graph $H$ has no isolated vertices. 
The edge set of the graph $H$ is in 
bijection with the edge set of the graph $G$. 

\begin{sublem}\label{g_cycle}
 For every vertex $g\in G$ its atomic cut forms a cycle in the graph $H$. 
\end{sublem}

\begin{proof}
Let $a(g)$ be the atomic cut of a vertex $g\in G$. By construction, the characteristic vector of 
the edge set $a(g)$ is orthogonal to the vector space $S$, as this is a subspace of the cycle space 
of $G$. 
As the set $S$ is locally connected by assumption, no (characteristic vector of a) proper nonempty 
subset of the atomic cut $a(g)$ is orthogonal to the vector space $S$. 

As the bonds of the graph $H$ are given by the 
minimal non-empty supports of vectors of the vector space $S$, the characteristic vector of the 
edge set $a(g)$ is in the cycle space of the graph $H$, and no proper nonempty subset of it is. 
Hence the edge set $a(g)$ is a circuit of the matroid $M^*$, and thus the edge set of a cycle of 
the graph $H$. 
\end{proof}

Now we construct an embedding of the graph $H$ into a 
pseudo-surface; here a \emph{pseudo-surface} is obtained from a compact 2-manifold\footnote{The 
relation between 2-manifolds and surfaces is as follows. Surfaces are 
connected compact 
2-manifolds.} by identifying finitely many points, see \autoref{fig:p-surface}.

   \begin{figure} [htpb]   
\begin{center}
   	  \includegraphics[height=3cm]{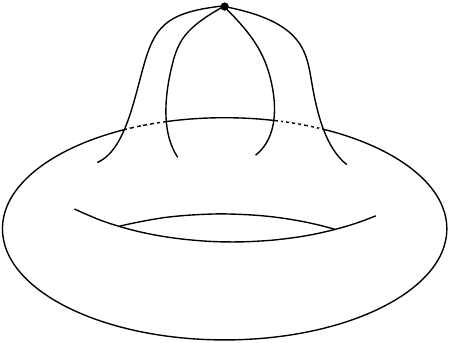}
   	  \caption{A pseudo-surface obtained from the torus by identifying two points. 
}\label{fig:p-surface}
\end{center}
   \end{figure}

Starting with the geometric realisation of the graph $H$, for each vertex $g\in G$ we attach a disc 
at its incident edges, which form a cycle of the 
graph $H$ by \autoref{g_cycle}. 
By $\Sigma$ we denote the topological space obtained from the graph $H$ by attaching these discs. 

\begin{sublem}\label{h_cycle}
 A neighbourhood around any vertex of $H$ in $\Sigma$ is obtained from finitely many disjoint 
discs by identifying their midpoints. 
\end{sublem}

\begin{proof}
Let $h$ be a vertex of the graph $H$. 
We construct an auxiliary graph whose vertex set is the set of edges incident with the vertex $h$. 
Two of these edges are adjacent in the auxiliary graph if there is a disc attached at both edges. 
As in the construction of the topological space $\Sigma$ all discs are attached at cycles of the 
graph $H$ each disc containing the vertex $h$ gives rise 
to precisely one edge of the auxiliary graph. 

As each edge of the graph $H$ is an edge of the graph $G$, it is contained in precisely two of the 
attached discs. Thus every vertex of this auxiliary graph has degree two. Hence this auxiliary 
graph is a vertex-disjoint union of cycles.
By construction of the graph $H$, the vertex $h$ is incident with an edge and hence this auxiliary 
contains at least 
one cycle.\footnote{We remark that \autoref{h_cycle} entails the statement that there is always at 
least one disc as 
neighbourhoods of a point can never be empty. } 

Hence a neighbourhood around any vertex of $H$ in $\Sigma$ is obtained from finitely many disjoint 
discs by identifying their midpoints. 
\end{proof}

As each edge of the graph $H$ is also an edge of the graph $G$, it is contained in precisely two of 
the attached discs. Hence by \autoref{h_cycle}, the topological space $\Sigma$ is a 
pseudo-surface and all identification points of $\Sigma$ are 
vertices of the graph $H$. Let $\Lambda$ be the unique compact 
2-manifold such that the 
pseudo-surface $\Sigma$ is a quotient of $\Lambda$ obtained by identifying finitely many points. 

We construct a graph $K$, which has the graph $H$ as a quotient as follows, and that embeds in the 
compact $2$-manifold $\Lambda$. 
For that, we locally cut\footnote{See \cite{loc2sepr} for a definition.} each vertex $h$ of $H$ 
that is 
an identification point of the pseudo-surface 
$\Sigma$ into 
one copy for each of its preimages in the compact 2-manifold $\Lambda$. A new vertex obtained by 
locally 
cutting $h$ is incident with those edges incident with $h$ that are embedded in the corresponding 
disc (formally, the edges incident with the vertex $h$ are partitioned into discs by 
\autoref{h_cycle}, and we have one copy of the vertex $h$ for each of these discs. This copy is 
incident with the edges meeting this disc).

Now we construct an embedding of the graph $G$ into the compact 2-manifold $\Lambda$. We embed each 
vertex of 
the graph $G$ into its disc. Then we embed half of each edge incident with a vertex $v$ 
within the 
disc for $v$ as a path joining $v$ and the midpoint of the edge of $H$ corresponding to that edge 
within the disc for $v$. The two halves of an edge meet at this midpoint of the edge of $H$. Hence 
this defines an 
embedding of the graph $G$.

Note that in the compact 2-manifold $\Lambda$, the graphs $G$ and $K$ are duals. 
By construction the dual graph $K$ of the embedding into the compact 2-manifold $\Lambda$ 
has the graph $H$ as a quotient; and this graph $H$ has the cycle matroid $M^*$, which in turn is 
the dual matroid of the $S$-local matroid for $G$. 
As the graph $G$ is connected by assumption, the 2-manifold $\Lambda$ is connected, and 
hence a surface.

Having completed, the construction of the embedding of the graph $G$ into the surface $\Lambda$, we 
 show that this embedding is $S$-facial. 
Let $o$ be a cycle of the graph $G$ that is in the set $S$. 
As the set $S$ is a subspace of the cycle space of the graph $G$, the edge set $o$ 
is a circuit of the $S$-local matroid for $G$. 
By construction of the graph $H$, 
the edge set $o$ is a bond of the graph $H$. 
So $o$ is a cut of the graph $K$, as it has $H$ as a quotient. 
So $o$ considered as a cut of $K$ is generated by the atomic cuts of $K$. As these atomic cuts are 
the faces of the embedding of $G$ in $\Lambda$, the cycle $o$ of the graph 
$G$ is generated by 
the faces of the 
embedding with all coefficients in $\{0,1\}$; that is, it is facially generated.
Hence the constructed embedding of $G$ is $S$-facial.

Next, to see $k$-admissibility, 
consider the sum over all face boundaries of faces of the 
embedding of the graph $G$ in the 2-manifold $\Lambda$ over the field $k$. 
This sum\footnote{For the field $k=\Fbb_3$, the signs of these vectors come from the representation 
of the cycle 
matroid of the graph $H$. Indeed, this gives a representation of the cycle matroid of $K$ by taking 
for the vector at an atomic cut just the restriction of the corresponding vector for the graph $H$. 
} is 
equal to the sum over the atomic cuts of the graph $K$, and hence evaluates 
to zero. 
Hence the 
embedding of $G$ in $\Lambda$ is $k$-admissible.

Finally, we check the `Moreover'-part; that is, we show that all faces of the embedding of the 
graph $G$ into the 2-manifold $\Lambda$ are discs. These faces are indexed by the vertices 
of the dual graph $K$. Let $x$ be an arbitrary vertex of $K$. The boundary of the 
face for $x$ consists of those edges that in the graph $K$ are incident with the vertex $x$. By 
construction of the embedding of the graph $G$, this face boundary is a closed trail.
It remains to show that this closed trail traverses every vertex $v$ of $G$ at most once. By 
\autoref{g_cycle} the atomic cut at $v$ intersects the atomic cut at $x$ in at most two edges. Thus 
the face boundary for $x$ can traverse the vertex $v$ at most once. Hence all faces of the 
embedding of $G$ are discs. 
\end{proof}

\section{Fencing in of holes}

In this section we prove \autoref{is_fencing_in} below, which is used in the proof of the main 
result mentioned in the Introduction. This lemma says that the set $S_r$, defined in the next 
sentence, 
is fencing in.
Given a graph $G$, we denote by $S_r$ the vector space generated by the 
cycles of $G$ of length at most $r$.

Throughout this section we abbreviate `area bounded 
by $S_r$' simply by `area'. 
And we fix a graph $G$ that is $S_r$-locally embedded in a surface. 
Additionally, for $k=\Fbb_3$ assume that the embedding is in an oriented surface.

\begin{lem}\label{boundary_to_set}
 Let $A$ be a sum of faces with empty boundary.
 Then either $A=0$ or $A$ contains all faces with the same coefficient.
\end{lem}

\begin{proof}
Assume that the sum $A$ is non-zero at some face $f$. 
As the boundary is identically zero, every face adjacent to $f$ must appear in the sum $A$.
Continuing like this inductively, we deduce that $A$ contains all faces, as surfaces are 
connected. 
 
 If $k=\Fbb_2$, then we are done. So we may assume that $k=\Fbb_3$. 
 Here we have the additional assumption that graph embeds in an oriented surface. Hence 
the face boundaries containing each edge 
traverse that edge in opposite direction. Using this, we do not only get that all faces must appear 
in the sum $A$ but also that they must appear in the sum with the same coefficient.  
\end{proof}

We say that an area is \emph{good} if its boundary is a geodesic cycle of $G$ of length at most 
$r$. 
Recall that an area is a sum of faces with all coefficients equal to plus one such that its 
boundary is in the vector space $S_r$. 
The sum of two areas whose faces are disjoint is again an area. Similarly, if one area $A$ includes 
an area $B$, then $A-B$ is an area.

\begin{lem}\label{sum_of_good}
 Every area is a sum of good areas with coefficients in $\{-1,+1\}$.
\end{lem}

\begin{proof}
 Let $A$ be an area and let $\alpha$ be its boundary. As $\alpha$ is in the vector space $S_r$ by 
definition, it is generated by cycles of length at most $r$, which in turn are generated by 
geodesic cycles of length at most $r$. 
Let $\alpha_1,..,\alpha_n$ be such a generating set. 
As the embedding is $S_r$-local, each cycle $\alpha_i$ bounds an area. 
Let $B$ be the sum of all these areas over the field $k$. 
The boundary of the sum $B$ is the vector 
$\alpha$.

Consider the sum $A-B$. Its boundary is identically zero. There are precisely two ways 
in which this could happen.

Firstly, $A=B$ and we are done.
Otherwise, by \autoref{boundary_to_set} the sum $A-B$ consists of all faces with the same 
coefficient. 

If $k=\Fbb_2$, modify the area bounded by $\alpha_1$ by adding $A-B$. This is again an area with 
the same boundary.
Then the resulting modified term for $B$ is equal to $A$, and we are done.

Hence it remains to consider the case $k=\Fbb_3$. 
Let $X$ be the sum over all areas with coefficient one. As the sum $X$ is equal to $A-B$ up 
to a constant, $X$ has empty boundary.  
Let $A_i$ be the area bounded by $\alpha_i$ chosen above.
Then $A_i'=X-A_i$ is an area whose boundary $-\alpha_i$. We distinguish two cases. 

{\bf Case 1:} $X=-(A-B)$.
Let $B'$ be the sum of the areas $A_i$ with `$-A_1'$' in place of the area `$A_1$'. 
Then $B'=B-X$. So $B'=A$, which gives the desired representation of $A$ as a sum. 

{\bf Case 2:} not Case 1. Then $X=A-B$. 
As we were done otherwise immediately, we assume that $n\geq 2$. 
Let $B'$ be the sum of the areas $A_i$ with `$-A_i'$' in place of the area `$A_i$' for $i=1,2$.
Then $B'=B-2X$. So $B'=A$, which completes the proof. 
\end{proof}

\begin{cor}\label{good_hole}
 Assume there is an area containing a hole. Then there is a good area containing a 
hole. 
 \end{cor}

\begin{proof}
 Let $A$ be an area containing a hole $h$. By \autoref{sum_of_good} the area $A$ is a sum of good 
areas. Hence one of these summands must contain the hole $h$. 
\end{proof}

\begin{lem}\label{topo_reroute1}
 Let $G$ be a graph embedded in a surface $\Sigma$. Let $\ell$ be a natural number. 
 Let $A$ be a closed subset\footnote{This just says that $A$ is an area. We write it this way to 
stress that here we consider $A$ as a topological object with its boundary rather than a 
combinatorial object.} of $\Sigma$ whose boundary is a geodesic cycle of $G$ of length at most 
$\ell$.
 Let $o$ be a cycle included in $A$ of length at most $\ell$. 
 Then there is a family of geodesic cycles of 
length at most $\ell$ included in the set $A$ such that their sum evaluates to $o$ over $\Zbb$.
\end{lem}

\begin{proof}
We denote the bounding cycle of the set $A$ by $\alpha$. 
Clearly the cycle $o$ is a sum of geodesic cycles of length at most $\ell$. 
Let $\Fcal$ be the subfamily of these generating cycles that  either are included in the closed 
subset $A$, or intersect its boundary 
$\alpha$.
As the boundary $\alpha$ is a geodesic cycle, paths outside the set $A$ whose endvertices are on 
the 
cycle $\alpha$, can be replaced by subpaths of the cycle $\alpha$ that are not longer. Now we 
manipulate all the geodesic cycles in the family $\Fcal$ in this way.
Consider the sum of this modified family. This sum is included in the set $A$. And in the interior 
of $A$ it coincides with the cycle $o$. So its difference to $o$ is supported at the boundary 
$\alpha$ of $A$. As this difference must be in the cycle space, it must be a multiple of the 
boundary 
cycle $\alpha$. Hence subtracting a suitable multiple of $\alpha$ from the modified family gives a 
way to write the cycle $o$ as a sum of geodesic cycles of length at most $\ell$ that are included 
in $A$. 
\end{proof}

We say that an area $A$ \emph{includes} an area $B$ if
if all faces contained in $B$ are also contained in $A$; that is, the sum $B$ is a subsum of the 
sum $A$. 

\begin{lem}\label{topo_reroute2}
 Let $A$ be a good area.
 Let $o$ be a cycle included in $A$ of length at most $r$. 
 Then there is an area included in $A$ bounded by $o$.
 \end{lem}

\begin{proof}
 As the cycle $o$ has length at most $r$ and the embedding is $S_r$-local, the cycle $o$ is the 
boundary of an area. Denote that area by $B$. 
Let $C$ be the sum (with coefficient one) of those
 faces of $B$ that are not included in the area $A$. 
As the boundary of the area $B$ is included in the area $A$, also the boundary of the sum $C$ must 
be included in the area $A$. Thus the boundary of the sum $C$ can only take nonzero values at the 
bounding cycle of the area $A$. As the only such vectors in the cycle space are multiples of 
the bounding cycle of the area $A$, the boundary of the sum $C$ has the form $x\cdot \alpha$, 
where $x$ is a coefficient in the field  $k$ and $\alpha$ is the bounding cycle of the area 
$A$.

If $x=0$, then the area $B-C$ is bounded by $o$, and we are done.

If $x=-1$, then the sum $X=A-(B-C)$ has the boundary $-o$; that is, the 
cycle $o$ with the reverse orientation. So  $X$ is the desired area.

Hence it remains to show that the case $x=1$ and $k=\Fbb_3$ is not possible. 
Then the sum $X=A-C$ has empty boundary. 
As the set $A$ is nonempty and disjoint from $C$ by construction, by \autoref{boundary_to_set} the 
sum $X$ takes the same non-zero value at all faces. As the area $A$ is bounded by a cycle, there is 
a face in $A$ and one outside. The sum $X$ takes the value plus one at the face inside and the 
value 
minus one outside. This is a contradiction, completing the proof.
\end{proof}

\begin{lem}\label{curve_in_surface}
 Let $\Sigma$ be a surface and let $B$ be a closed subset whose boundary $b$ is a circle.
 Let $\gamma$ be a curve intersecting the circle $b$ in finitely many points but not in its 
endvertices.
 Then precisely one of the endpoints of the curve $\gamma$ is in the set $B$ if and only if the 
curve $\gamma$ intersects the cycle $b$ oddly.  
\end{lem}

\begin{proof}[Proof:]
 by induction on the number of intersection points of the curve $\gamma$ and the boundary $b$.
\end{proof}

\begin{lem}\label{minimal_good}
  Assume there is a good area containing a hole. Then there is a good area that is fencing in.
\end{lem}

\begin{proof}
 By assumption, there is a good area containing a hole. 
 Pick a good area containing as few holes as possible but at least one.
 Call it $A$. It remains to prove that the area $A$ is fencing in.

Suppose not for a contradiction. Then there is an area $C$ such that it contains a hole of the 
area $A$ and avoids a hole of $A$. 
\begin{sublem}\label{is_goood}
 There is a good area $B$ that contains a hole of $A$ and avoids a hole of $A$. 
\end{sublem}

\begin{proof}By \autoref{sum_of_good} the 
area $C$ is a sum of good areas. 
Let $C'$ be the sum of such good areas that contain some hole of $A$. 
Suppose for a contradiction that all summands of $C'$ take the same coefficient at all 
holes of $A$. 
Then in
the sum $C'$ all holes appear with the same multiplicity. By construction, each holes of $A$ 
appears with the same multiplicity in the sums $C'$ and $C$. This is a contradiction as in $C$ two 
holes of $A$ have a different multiplicity.
Hence there is a good area that has different multiplicities at holes of $A$. As areas 
can only have multiplicity one and zero at holes, there is a  a good area that contains a hole 
of $A$ and avoids a hole of $A$. 
\end{proof}

By \autoref{is_goood}, there is a  good area $B$ that contains a hole of $A$ and avoids a hole of 
$A$. 
Pick a hole $h_1$ that is in both areas $A$ and $B$, and a hole $h_2$ that is only in the area $A$. 
As the area $A$ is bounded by a cycle, it is connected.
So there is a curve $\gamma$ included in the closed set $A$ from the hole $h_1$ to the hole $h_2$ 
such that 
it does not intersect the bounding cycle of the good area $A$. By modifying the curve $\gamma$ 
locally if necessary, we may assume, and we do assume, that $\gamma$ intersects the bounding cycle 
of the good area $B$ only in finitely many points. 

Denote the bounding cycle of the good area $A$ by $a$.
Denote the bounding cycle of the good area $B$ by $b$.

\begin{sublem}\label{a_b_inter}
 The cycles $a$ and $b$ intersect in at least two points.
\end{sublem}

\begin{proof}
As the curve $\gamma$ joins a hole in $B$ with a hole outside, the curve $\gamma$ has to intersect 
the cycle $b$; in particular $b$ contains a point in the interior of the area $A$. 
Suppose for a contradiction that the cycle $b$ is 
included in the area $A$. Then by \autoref{topo_reroute2}, there is an area $B'$ included in the 
area  $A$ 
and bounded by $b$. 
Consider the sum $B-B'$. This sum has empty boundary. 
By \autoref{boundary_to_set}, either this sum is identically zero or it contains all faces with 
the same multiplicity. 
So either the areas $B$ and $B'$ agree or they 
partition the 
faces.
In either case, the area $B'$ contains fewer holes than $A$ but precisely one of $h_1$ and $h_2$.
Note that the area $B'$ is good. 
This is a contradiction to the choice of the area $A$. 

So the cycle $b$ contains a 
point outside the area $A$. Applying \autoref{curve_in_surface} to two subpaths of the cycle $b$ 
between vertices inside $A$ and outside $A$, yields that the cycles $a$ and $b$ have to intersect 
in at least two points. 
\end{proof}

\begin{sublem}\label{P_exists}
 There is a subpath $P$ of the cycle $b$ that intersects the cycle $a$ precisely in 
its endvertices and that intersects the curve $\gamma$ in an odd number of points. 
\end{sublem}

\begin{proof}
For each intersection point $y$ of the curves $b$ and $\gamma$ pick the unique subpath of the cycle 
$b$ that contains the point $y$ and that intersects the cycle $a$ precisely in its endvertices; 
this is well defined by \autoref{a_b_inter}. 
Denote that path by $P_y$. Two paths $P_y$ are either identical or intersect at most in their 
endvertices. If all paths $P_y$ intersected the curve $\gamma$ evenly, then the cycle $b$ 
would intersect the curve $\gamma$ evenly. This is not possible as precisely one of the holes 
$h_1$ and $h_2$ is in the area $B$ by \autoref{curve_in_surface}. Hence one path $P_y$ must 
intersect the curve $\gamma$ oddly. Pick such a path for $P$. 

\end{proof}

Let $P$ be a path as in \autoref{P_exists}. 
Denote the two endvertices of the path $P$ by $x$ and $y$. Let $Q$ be a shortest path between the 
vertices $x$ and $y$ included in the geodesic boundary cycle $a$ of the good area $A$. 
As the length of the geodesic path $Q$ cannot be longer than any $x$-$y$-path included in the 
bounding cycle $b$ of the good area $B$, the length of the cycle $PQ$ is at most the length of the 
bounding cycle of $B$; and thus at most $r$.

As the path $P$ contains an interior point of the area 
$A$, and it intersects its boundary only in its endvertices, the connected path $P$ is included in 
the area $A$. So the cycle $PQ$ is included in the area $A$. 

By \autoref{topo_reroute1}, the cycle $PQ$ can be written as a sum of geodesic cycles of length at 
most $r$ that are included in $A$. As the cycle $PQ$ intersects the curve oddly, one of the 
generating geodesic cycles has to intersect the curve $\gamma$ oddly. Pick such a geodesic cycle 
and denote it by $x$.

By \autoref{topo_reroute2} there is an area $X$ included in the area $A$ whose boundary is the 
geodesic cycle $x$. By construction the area $X$ is good.
As the curve $\gamma$ intersects the cycle $x$ oddly, the area $X$ contains precisely one of the 
holes $h_1$ and $h_2$ by \autoref{curve_in_surface}. Thus the area 
$X$ contains a hole. As all holes of the area $X$ are holes of the area $A$ but $X$ does not 
contain 
the hole $h_1$ or $h_2$, the existence of the good area $X$ is a contradiction to the choice of the 
area $A$. 
Thus the area $A$ must be fencing in. 
\end{proof}

We introduce the following notation to state the next lemma for the fields $\Fbb_2$ and $\Fbb_3$ 
simultaneously.
An \emph{$\Fbb_3$-oriented} embedding is an embedding in an oriented surface. Every embedding of a 
graph in a surface is \emph{$\Fbb_2$-oriented}.  

\begin{lem}\label{is_fencing_in}
The vector space $S_r$ generated by the cycles of length at most $r$ over $k$ is fencing in (for 
every 
$S_r$-facial $k$-oriented embedding). 
\end{lem}

\begin{proof}
We prove this by induction on the number of holes.
The induction starts where there is no area containing holes; then the statement is true 
as there 
are no minimal areas.
So suppose there is an area containing holes.
By \autoref{good_hole}, there is also a good area that contains a hole. 
By \autoref{minimal_good} there is an area that is bounded by a geodesic cycle and is fencing in. 
Denote that area by $A$ and its bounding cycle by $\alpha$. 

Now construct a new embedding from the old embedding by deleting all faces 
contained in the area $A$ -- including all vertices of the graph $G$ that are in the interior of 
the area $A$ -- and attaching a disc at the cycle $\alpha$. 
Denote the new embedded subgraph of $G$ by $G'$. 

As the area $A$ contains a hole, this new embedding of this subgraph has strictly fewer holes. 
We denote the vector space generated by cycles of length at most $r$ in the graph $G'$ by 
$S_r'$. 
Hence by induction, the set $S_r'$ for this new embedded graph $G'$ is fencing in. 

Areas of the graph $G'$ are also areas of the graph $G$ after removing the newly added disc with 
boundary $\alpha$ of the embedded graph $G'$ if necessary. 
The following lemma allows us to transform areas of the graph $G$ into areas of the subgraph 
$G'$. Here we remark that in \autoref{transfer_principle} the `$+$' 
denotes addition of areas in the graph $G$. In particular, in the 
corresponding term we add the face $A$ as a face of the graph $G$ with all its holes. 

\begin{sublem}\label{transfer_principle}
Assume there is a face not contained in the area $A$. 
 For every area $B$ of $G$, there is an area $B'$ that is an area for the graphs $G$ and $G'$ 
such that one of $B'$ or $B'+A$ contains 
the same holes as the area $B$. 
\end{sublem}

\begin{proof}
Let $\hat B$ be the sum (with coefficients one) over all faces of the area $B$ that are contained 
in the area $A$. We want to show that the sum $\hat B$ is an area; that is, its boundary is in the 
set $S_r$. 
Denote the boundary of the area $B$ by $b$. 

Our first step will be to construct from the boundary $b$ a candidate 
for the boundary of the sum $\hat B$, as follows.
As the boundary $\alpha$ of the area $A$ is a geodesic cycle, paths outside the set $A$ whose 
endvertices are on 
the 
cycle $\alpha$, can be replaced by subpaths of the cycle $\alpha$ that are not longer.
Applying this to the boundary $b$, gives a closed walk $\gamma$ of length at most $r$ that agrees 
with the boundary $b$ on the interior of the area $A$ and is included in $A$. 
By \autoref{topo_reroute2}, there is an area $C$ bounded by $\gamma$ and included in $A$. 

Now consider the difference $\hat B-C$. Its boundary is a difference of the boundaries of $\hat B$ 
and $C$. As these two boundaries are included in $A$ and agree on the interior of $A$, the boundary 
of $\hat B-C$ is supported on the boundary of the area $A$. As the boundary is in the cycle space, 
it must be a multiple of the bounding cycle $\alpha$ of $A$.
Thus there is a coefficient $x$ in $k$ such that
$\hat B-C+x\cdot A$ has empty boundary. As there is a face outside the area $A$ by assumption, this 
sum is not supported on all faces. Hence by \autoref{boundary_to_set}, we have that
\[
 \hat B-C+x\cdot A=0
\]

If $k=\Fbb_2$, then $\hat B=C-x \cdot A$ is an area. Then $B'=B-\hat B$ is an area that uses no 
faces of the area $A$. As the area $A$ is fencing in, the area $\hat B$ contains either all holes 
of $A$ or none. So $B'$ has the desired properties. 
So it remains to consider the case that $k=\Fbb_3$. 

If $x=0$, then $\hat B=C$, so $\hat B$ is an area. Arguing as above, we conclude that  
$B'=B-\hat B$ is the desired area.

If $x=1$, then $-\hat B= A-C$. As $A-C$ is an area and the sum $\hat B$ has never the coefficient 
minus one, we conclude that $\hat B=0$. So $B'=B$ is the desired area.

If $x=-1$, then $\hat B= A+C$. Then faces contained in the area $C$ have coefficient minus one on 
the right. As $\hat B$ has no such coefficients, the area $C$ must be empty. So $\hat B= A$. 
Then $B'=B-A$ is the desired area.
\end{proof}

Having finished the proof of \autoref{transfer_principle}, we show that $S_r$ is fencing in.
Let $B$ be a minimal area. 
If the area $B$ has the same holes as the area $A$, then it is fencing in, and we are done.
So we may assume, and we do assume, that the area $B$ contains a hole that is not 
contained in the minimal area $A$.
As the area $B$ is minimal, it cannot contain any hole of the area $A$, as $A$ is fencing in. 
By \autoref{transfer_principle}, there is an area $B'$ of $G'$ such that one of $B'$ or $B'+A$ 
contains the same holes as the area $B$. 
Let $C'$ be a minimal area of $G'$ containing a hole of the area $B$. By induction, the 
area $C'$ is fencing 
in for $G'$. In particular, all its holes are contained in the 
area $B'$. Let $C''$ be the area obtained from $C'$ by setting the coefficient of the newly added 
disc with 
boundary $\alpha$ of the embedded graph $G'$ to zero. Note that $C''$ is an area of the graph $G$ 
with the same holes as the area $C'$ of $G'$. 

So the area $C''$ contains no hole outside the 
minimal area $B$. By minimality of $B$, the areas $B$ and $C''$ contain the same holes. 
By replacing the area `$B$' by `$C''$' if necessary, we assume that $B=B'$ and $B'$ is fencing 
in 
for the graph $G'$. 

Now let $D$ be an arbitrary area of the graph $G$.
By \autoref{transfer_principle}, there is an area $D'$ of $G'$ such that one of $D'$ or 
$D'+A$ contains the same holes as the area $D$.
As the area $B'$ is fencing in for $G'$, the area $D'$ either contains no hole of $B'$ or all holes 
of $B'$. 
As the area $B$ contains no hole of the area $A$, the area $D$ either contains no hole of $B$ or 
all holes of the area $B$. As the area $D$ was arbitrary, the area $B$ is fencing in. This 
completes 
the induction step.
\end{proof}

\section{Abstract duality for surfaces}

In this section we deduce the main result stated in the Introduction from the lemmas proved in the 
previous sections. 

First we prove the following analogue of Whitney's abstract duality theorem for 
general surfaces.

\begin{thm}\label{abstract_duality}
 Let $G$ be a $2$-connected graph, and $S$ be a nonempty subspace of its cycle space over 
$\Fbb_2$ with a cyclic 
generating set.
 Assume $S$ is locally connected and fencing in.
 
 Then $G$ has an $S$-facial embedding into a surface if and only if the $S$-local matroid for $G$ 
over $\Fbb_2$ 
is co-graphic.
\end{thm}

\begin{proof}
Since the set $S$ is 
fencing in and has a nonempty cyclic generating set by assumption, by 
\autoref{duality_in_surfaces}, if $G$ 
has an $S$-facial embedding into a surface, then the $S$-local matroid for $G$ 
is co-graphic. 

Since the set $S$ is locally connected by assumption, by \autoref{construct_embedding} if the 
$S$-local matroid for $G$ 
is co-graphic, then 
$G$ has an $S$-facial embedding.
\end{proof}

We say that an embedding of a graph $G$ in a (pseudo-) surface is $r$-locally planar if 
all cycles of length at most $r$ are generated by the face boundaries over the field $k$. Note that 
a graph is 
$r$-locally planar 
if and only if 
it is $S_r$-facial, where $r$ is the vector space generated by the cycles of $G$ of length at most 
$r$. 

\begin{cor}\label{r-local_duality}
 Any $r$-locally 2-connected graph $G$  has an 
$r$-locally 
planar embedding if and only if the 
$r$-local matroid for $G$ over $\Fbb_2$ is 
co-graphic.
\end{cor}

\begin{proof}
It suffices to prove this theorem for every connected component of the graph $G$. Hence we may 
assume, and we do assume, that the graph $G$ is connected, and so $2$-connected by the local 
connectivity assumption. 
If the graph $G$ has no cycle of length at most $r$, any embedding is $r$-locally planar, so there 
is nothing to prove. So we assume that the graph $G$ has a cycle of length at most $r$. 

 Let $S$ be the set of all edge sets generated by the cycles of the graph $G$ of length at most 
$r$. By assumption the set $S$ is nonempty. 
Clearly $S$ is a subspace of the cycle space of $G$ with a cyclic generating set. 
By \autoref{is_fencing_in}, the set $S$ is fencing in. 
As $G$ is $r$-locally $2$-connected, by \autoref{is_loc_con}, the set $S$ is 
locally connected. By 
definition, an $S$-facial embedding 
is simply an $r$-locally planar embedding. The $r$-local matroid is identical to the 
$S$-local 
matroid. 

Hence by \autoref{abstract_duality}, the graph $G$ has an $r$-locally planar embedding if and 
only 
if the $r$-local matroid for $G$ is 
co-graphic.
\end{proof}

Actually, \autoref{r-local_duality} can be extended beyond the locally 2-connected case. In 
order to prove this, we do some preparation. 
A \emph{pseudo-manifold} is a topological space obtained from a compact 2-dimensional manifold by 
identifying finitely many points.
\begin{eg}
 A connected pseudo-manifold is a pseudo-surface. 
\end{eg}

We refer to the finitely many identification points of a pseudo-manifold as 
\emph{singularities}.
Given a graph $G$ embedded in a pseudo-manifold $\Sigma$, a vertex $v$ embedded to a singularity of 
$\Sigma$ is an \emph{$r$-local singularity} if the point $v$ in $\Sigma$ has a neighbourhood $O$ 
such that two edges incident with $v$ intersect the same connected component of $O-v$ if and only 
if their endvertices aside from $v$ are in the same connected component of the punctured ball 
$B_{r/2}(v)-v$. 
An embedding of a graph in a pseudo-manifold is \emph{$r$-nice} if its singularities are 
all $r$-local singularities.

\begin{lem}\label{cutting_and_embeddings}
Let $G$ be a graph and $G'$ be the graph obtained from $G$ by $r$-locally cutting a 
vertex $v$ of $G$.
 Then $G$ embeds $r$-locally planarly $r$-nicely in a pseudo-manifold with $v$ being an $r$-local 
singularity if and only if $G'$ embeds $r$-locally planarly $r$-nicely in a pseudo-manifold.
\end{lem}

\begin{proof}
First assume that  $G'$ embeds $r$-locally planarly $r$-nicely in a pseudo-manifold. We obtained an 
embedding 
of the graph $G$ by identifying all slices of the vertex $v$. This defines an embedding of the 
graph $G$ in a pseudo-manifold with $v$ being an $r$-local singularity. 
This embedding is $r$-nice by construction. 
This embedding is 
$r$-locally planar as every non-facially generated cycle of $G$ 
is either a non-facially generated cycle of $G'$ or not a cycle of $G'$ at all.
In the first case, the cycle has length more than $r$ by the $r$-local planarity of the 
embedding of $G'$. 
Cycles of the second type have length 
at least $r+1$ by definition of $r$-local cutting. So the 
embedding of $G$ has the 
desired properties. 
 
Next assume that  $G$ embeds $r$-locally planarly $r$-nicely in a pseudo-manifold with $v$ being an 
$r$-local 
singularity. We obtain an embedding of the graph $G'$ by replacing the vertex $v$ by its slices 
(formally we delete the vertex $v$ and then take the completion of the remainder. 
The newly 
added points are precisely the slices of $v$). As the vertex $v$ is an  $r$-local 
singularity, this defines an embedding of the graph $G'$. 
This embedding is $r$-nice by construction. 
As every non-facially generated cycle of the 
graph $G'$ is a non-facially generated cycle of the graph $G$, all these cycles have more than 
$r$.
Hence this embedding is $r$-locally planar. 
 \end{proof}

\begin{cor}\label{go_to_blocks}
 A graph $G$ has an $r$-locally planar $r$-nice embedding in a pseudo-surface if and only 
if all its $r$-local blocks have 
$r$-locally planar embeddings in surfaces. 
\end{cor}

\begin{proof}
 This statement with `embeddings in surfaces' replaced by `$r$-nice embeddings in 
pseudo-surfaces' follows directly from \autoref{cutting_and_embeddings} by induction on the 
number of $r$-local 
cutvertices. 
Finally note that as $r$-local blocks do not have $r$-local cutvertices, $r$-nice 
embeddings of such graphs in pseudo-surfaces are always embeddings in genuine surfaces. 

\end{proof}

Next we prove the following extension of \autoref{r-local_duality}. 

\begin{thm}\label{r-local_duality2}
Any graph $G$ has an $r$-locally planar $r$-nice embedding in a 
pseudo-surface 
if 
and 
only if the 
$r$-local matroid for $G$ over $\Fbb_2$ is 
co-graphic.
\end{thm}

\begin{proof}
Extending the well-known block-cutvertex theorem, in \cite{loc2sepr} it 
is shown that 
any graph has an edge-disjoint decomposition into its $r$-local blocks.
Each of these $r$-local blocks is $r$-locally $2$-connected. Moreover, the $r$-local matroid 
of $G$ 
is the direct sum over the $r$-local matroids of the blocks; indeed, analogous to the 
fact that cutting at cutvertices does not change the cycle matroid of $G$, locally cutting at 
$r$-local 
cutvertices does not change the $r$-local matroid of $G$. 

Applying \autoref{r-local_duality} to each $r$-local block separately yields that 
all $r$-local blocks have an $r$-locally planar embedding if and only if the $r$-local 
matroid for 
$G$ is 
co-graphic. 
So the statement follows by \autoref{go_to_blocks}.
\end{proof}

\begin{proof}[Proof of \autoref{mainthm_whitney_intro}.]
In the Introduction we abbreviated the `$r$-local matroid for $G$ over $\Fbb_2$' simply by the 
`$r$-local matroid for $G$'. 
The other difference between \autoref{mainthm_whitney_intro} and \autoref{r-local_duality2} is that 
the term `$r$-nice' is omitted in the statement in the Introduction. 
It is an easy exercise to make an embedding in a pseudo-surface $r$-nice (first 
note that the only vertices that can be mapped to singularities are $r$-local cutvertices. Then 
take 
the rotation system of the old embedding and at each $r$-local cutvertex $v$ replace the rotator by 
one rotator for every slice of $v$. This rotator at a slice is obtained from the original 
rotator at $v$ by restricting to the edges incident with the slice. This defines an embedding in 
a pseudo-surface that is $r$-locally planar and $r$-nice by construction.)

So \autoref{mainthm_whitney_intro} is just a restatement of 
\autoref{r-local_duality2}.
\end{proof}

Analogous to \autoref{abstract_duality} and \autoref{r-local_duality2} we prove the following 
oriented analogues.

\begin{thm}\label{abstract_duality_or}
 Let $G$ be a $2$-connected graph, and $S$ be a nonempty subspace of its cycle space over 
$\Fbb_3$ with a cyclic 
generating set.
 Assume $S$ is locally connected and fencing in.
 
 Then $G$ has an $S$-facial embedding into an orientable surface if and only if the $S$-local 
matroid for $G$ over $\Fbb_3$
is co-graphic.
\end{thm}

\begin{proof}
 The proof is the same as the proof for \autoref{abstract_duality}, where we take the field $k$ to 
be `$\Fbb_3$' in place of `$\Fbb_2$'. Additionally only note that a $k$-admissible embedding such 
that all faces are disc, defines an orientation of that surface. Hence the embedding $\iota$ 
constructed in \autoref{construct_embedding} is into an oriented surface. 
\end{proof}

\begin{thm}\label{r-local_duality2_or}
 Any graph $G$ has an $r$-locally planar $r$-nice embedding in 
an orientable pseudo-surface if and 
only if the 
$r$-local matroid for $G$ over $\Fbb_3$ is 
co-graphic.
\end{thm}

\begin{proof}
 The proof is the same as the proof for \autoref{r-local_duality2}, where we take the field $k$ to 
be `$\Fbb_3$' in place of `$\Fbb_2$'.
\end{proof}

\section{Concluding remarks}\label{sec:concluding}

An LEW embedding of a graph $G$ is an embedding in a surface such that all faces are strictly 
shorter than the non-contractible cycles of the embedding. 
Every LEW embedding is an $r$-locally planar embedding, where $r$ is the maximum length of a face 
of the embedding.
Locally planar embeddings are more general than LEW 
embeddings in the following ways.
Firstly, faces in locally planar embeddings can have arbitrary size. 
Secondly, while being contractible is a special case of being facially generated, the 
converse is not true. 
So from the point of view of algebraic topology, local planarity is phrased 
in terms of the more general homological notion, while LEW embeddings rely on the more restricted 
homotopical definition.

Of particular interest is the relation between locally planar embeddings and minimum genus 
embeddings. In \cite{thomassen1990embeddings} Thomassen, proved that LEW embeddings are minimum 
genus embeddings. 
\begin{eg}
As explained above, locally planar embeddings include LEW embeddings. Building on this, we 
construct a large class of locally planar embeddings that are not LEW embeddings but still of 
minimum genus. 
For that, take a LEW embedding of a graph $G$. Let $r$ be the maximum length of a face of this 
embedding. 
Pick a face $f$ of the embedding arbitrarily. Pick an arbitrary planar 
graph that has $f$ has a face, and $f$ is a geodesic cycle in that graph. Now glue that planar 
graph along the face $f$ onto the graph $G$. Call that new graph $H$. 
The embedding of the graph $G$ extends to an embedding of the graph $H$ by embedding the new planar 
graph within the face $f$. This embedding is $r$-locally planar. 

And it is a minimum genus embedding of the graph $H$, as the genus of $H$ cannot be smaller than 
that of the subgraph $G$.  
\end{eg}

\begin{oque}
 Can you characterise when locally planar embeddings are minimum genus embeddings?
\end{oque}

Our results provide a polynomial algorithm that computes for every graph $G$ the maximum value of 
$r$ such that $G$ has an $r$-locally planar embedding. Let $n$ denote the number of vertices of 
$G$. If $r\geq n$, then $r=\infty$. As $r$-local planar embedability is a monotone 
property, to compute the maximum value for $r$, we need to check $r$-local embeddability for 
$log(n)$ values (the first being $n/2$, then $n/4$ or $n/2+n/4$ etc). For each particular value of 
$r$ we first compute the $r$-local matroid. 
For that we have to construct in polynomial time a set of cycles of length at most $r$ generating 
all cycles of length at most $r$. This can be done as follows. For any two 
vertices $v$ and $w$ of $G$ of distance $d$, a maximum set of shortest $v$-$w$-path that are 
internally disjoint can be computed in polynomial time. Denote these paths by $P_1,..,P_m$.
Let $Q(v,w,d)$ be the set of cycles $P_1P_2, P_2P_3,...P_{m-1}P_m$. For any edge $e$ of $G$ such 
that there is a vertex $v$ of $G$ such that there are two shortest paths from $v$ to the 
endvertices of $e$ that are internally disjoint, denote by $o(e,v)$ the cycle composed of these two 
paths and the edge $e$. 
It is straightforward to show 
that the set of all the $Q(v,w,d)$ for $d\leq r/2$ and the set of these cycles $o(e,v,d)$ for 
$d\leq (r-1)/2$ generate all cycles of length at most $r$ and can be computed in polynomial time. 
Then it just remains to verify whether the $r$-local matroid is co-graphic, which can be done in 
polynomial time as mentioned in the Introduction. This completes the full description of the 
algorithm.

In the following we give a rough estimate how far locally planar embeddings are from minimum genus 
embeddings. We denote by $def(G)$ the genus deficit of an $r$-locally planar embedding of $G$ 
(that is, the 
difference between the genus of the embedding and 
the optimal 
genus), by $\alpha$ the 
rank 
of the $r$-local matroid for $G$, by $g(G)$ the girth of $G$ and by $E$ the edge-number of $G$. 
\begin{obs}
Assume a connected graph $G$ has an $r$-locally planar embedding. Then:
 \[
def(G)\leq  \alpha+\frac{2E}{g(G)}-E-1
\]
\end{obs}

\begin{proof}
Fix an $r$-locally planar embedding of $G$. Let $H$ be the dual graph of that embedding. 
Let $F$ be the number of faces for that embedding; that is, the number of vertices of the graph 
$H$. 
By \autoref{duality_in_surfaces} and \autoref{is_fencing_in}, $H$ has a quotient $H'$ whose cycle 
matroid is the dual of the $r$-local matroid of $G$. So $|V(H')|=E-\alpha+1$. So $F\geq E-\alpha+1$.

Let $H^*$ 
be a dual graph of $G$ in some optimal genus embedding. Let $F^*$ be the number of faces for an 
optimal embedding. Using the definition of the average degree of the graph $H^*$, we estimate: $F^* 
\leq \frac{2E}{g(G)}$.

By Euler's Formula, the genus deficit is equal to $F^*-F$. Plugging in the two expressions for $F$ 
and $F^*$ gives the desired result. 
\end{proof}

Exploring classes of matroids that can arise as $S$-local matroids further is an exciting 
direction for future research. Indeed, inspired by a theorem of Seymour 
\cite{seymour1981recognizing} and their work on the Matroid Minors 
Project, Geelen, Gerards and Whittle introduced the class of quasi-graphic matroids 
\cite{geelen2018quasi}, see \cite{bowler2020describing} for further examples of quasi-graphic 
matroids. 
It seems natural to compare the conditions in \autoref{abstract_duality} to those for  
quasi-graphic matroids. In particular the 
following seems to be of interest.

\begin{oque}
 Are there natural classes of quasi-graphic matroids 
 that are cographic and $S$-local matroids for graphs $G$?
\end{oque}

\bibliographystyle{plain}
\bibliography{literatur}

\end{document}